\documentclass{article}
\usepackage[english]{babel}
\usepackage{amsrefs,amssymb,amsmath,amsthm,amsfonts,epsfig,graphicx,psfrag}
\usepackage{hyperref}
\hypersetup{
    colorlinks,
    citecolor=black,
    filecolor=black,
    linkcolor=black,
    urlcolor=black
}
\usepackage[all,cmtip]{xy}
\usepackage[utf8]{inputenc}
\DeclareMathOperator{\cov}{cov}
\newcommand{\lideal} {\lambda_{\rm{ideal}}}

\newcommand{\U} {\mathcal U}

\DeclareMathOperator{\ent}{h}
\DeclareMathOperator{\entf}{ent}
\DeclareMathOperator{\ucap}{cap}
\DeclareMathOperator{\card}{card}

\theoremstyle{plain}
\newtheorem{thm}{Theorem}[section]
\newtheorem{cor}[thm]{Corollary}
\newtheorem{lem}[thm]{Lemma}
\newtheorem{prop}[thm]{Proposition}

\theoremstyle{definition}
\newtheorem{df}[thm]{Definition}
\newtheorem{ex}[thm]{Example}
\newtheorem{rmk}[thm]{Remark}

\theoremstyle{remark}

\DeclareMathOperator{\dist}{d}

\DeclareMathOperator{\diam}{diam}

\DeclareMathOperator{\inte}{int}

\newcommand{\azul}[1]{\textcolor{black}{#1}}

\newcommand{\R}{\mathbb R}

\newcommand{\Z}{\mathbb Z}
\newcommand{\N}{\mathbb N}
\renewcommand{\epsilon}{\varepsilon}
\newcommand{\epsilona}{{\epsilon_0}}
\newcommand{\epsilonb}{{\epsilon_1}}
\newcommand{\epsilonc}{{\epsilon_2}}

\newcommand{\expc}{\xi}

\begin{document}

\author{Alfonso Artigue\footnote{Email: artigue@unorte.edu.uy. Adress: Departamento de Matemática y Estadística del Litoral, Universidad de la Rep\'ublica, Gral. Rivera 1350, Salto, Uruguay.}}
\title{Self-Similar Hyperbolicity}
\date{\today}
\maketitle

\begin{abstract}
In this paper we consider expansive homeomorphisms of compact spaces 
with a hyperbolic metric presenting a self-similar behavior on stable and unstable sets. 
Several application are given related to Hausdorff dimension, entropy, intrinsically ergodic measures 
and the transitivity of expansive homeomorphisms with canonical coordinates.
\end{abstract}
\section{Introduction} 


A homeomorphism $f\colon M\to M$ of a compact metric space $(M,\rho)$ 
is \emph{expansive} if there is $\expc>0$ such that 
if $x, y\in M$ are different then $\rho(f^n(x),f^n(y))>\expc$ for some $n\in\Z$. 
Examples of such dynamics are hyperbolic sets, in particular Anosov diffeomorphisms and 
basic sets of axiom A diffeomorphisms. 
By definition, expansivity is independent from hyperbolicity and smooth structures, 
it is a topological concept.
However, several authors constructed special hyperbolic metrics for this kind of dynamics, 
see for example \cites{CoRe,Reddy82,Reddy83,Fried,Sakai95,Fa,Dov,FKM}. 
In this paper we will study expansive homeomorphisms with a self-similar hyperbolic metric.

To motivate our first result we recall that
in \cite{Fa} Fathi constructed a compatible metric, that we will denote by $\dist_F$, 
for which $f$ is a Lipschitz isomorphism and 
there are 
$\expc>0$ and $\lambda>1$ such that 
$\max \{\dist_F(f(x),f(y)),\dist_F(f^{-1}(x),f^{-1}(y))\}\geq \lambda\dist_F(x,y)$
if $\dist_F(x,y)\leq\expc$. 
A metric as $\dist_F$ is called \emph{adapted} or \emph{Lyapunov hyperbolic}.
In \cite{Dov} Dovbish
obtained a hyperbolic metric $\dist_D$ 
with an asymptotic homothetic behavior on local stable and unstable sets. 
That is, there are two constants $0<\lambda_s,\lambda_u<1$ such that
$\dist_D(f(x),f(y))$ approximates $\lambda_s\dist_D(x,y)$ 
for $y\in W^s_\epsilon(x)$ as $\epsilon\to 0$; 
with analogous estimate on local unstable sets.
We recall that
for $\epsilon>0$ the \emph{local stable set} of $x\in M$ is 
\[
 W^s_\epsilon(x)=\{y\in M: \dist(f^n(x),f^n(y))\leq \epsilon\text{ for all }n\geq 0\}.
\]
The \emph{local unstable set} is defined as 
\[
 W^u_\epsilon(x)=\{y\in M: \dist(f^{-n}(x),f^{-n}(y))\leq \epsilon\text{ for all }n\geq 0\}.
\]

Our first result, Theorem \ref{teoPerfHyp}, states that for every expansive homeomorphism 
of a compact metric space there are a compatible metric $\dist$ and two constants $\expc>0$ and $\lambda>1$ such that 
if $\dist(x,y)<\expc$ then 
\[
  \max \{\dist(f(x),f(y)),\dist(f^{-1}(x),f^{-1}(y))\}= \lambda\dist(x,y).
\]
A hyperbolic metric with this property will be called \emph{self-similar}
and we say that $\expc$ is an \emph{expansive constant} and 
that $\lambda$ is an \emph{expanding factor} of the metric.
Obviously, Dovbysh's conditions hold for a self-similar metric without taking limits. 
The construction of a self-similar hyperbolic metric that we present follows standard techniques. 
In fact, it is Fathi's metric with only a small variation that is explained in Remark \ref{rmkConstFathi}.

In \cite{FKM}*{Problem 2.6} Fujita, Kato and Matsumoto ask: \emph{Do Positively expansive maps expand strictly small distances?}
In our terminology they ask about the existence of a self-similar hyperbolic metric 
for a positively expansive map.
A continuous map $f\colon M\to M$ of the compact metric space $(M,\rho)$ 
is \emph{positively expansive} if there is $\delta>0$ such that 
if $x\neq y$ then $\rho(f^n(x),f^n(y))>\delta$ for some $n\geq 0$.
They give a positive answer in the case of open positively expansive maps and 
for positively expansive maps of graphs. 
In Theorem \ref{thmSSPosExp} we give a positive answer with full generality, its proof is only sketched because it is analogous to the case of expansive 
homeomorphisms.

The rest of the paper is devoted to explore the consequences of the 
self-similarity of a hyperbolic metric. Let us describe the content of this paper while stating more results that we obtained. 
In \S \ref{secExpSelfSim}, besides proving Theorems \ref{teoPerfHyp} and \ref{thmSSPosExp}, 
examples are given and basic properties of these metrics are investigated. 

In \S \ref{secTopEnt}, Theorem \ref{teoEntCap}, 
we prove the equation 
\[
 \ucap(M,\dist)=\frac{\entf(f)}{\log(\lambda)}
\]
relating the capacity of the space, the entropy of the homeomorphism and the expanding factor of 
a self-similar metric. 
This is a fundamental equation of self-similar hyperbolicity that holds for every expansive homeomorphism of a compact metric space with a self-similar metric.
It was previously proved in \cite{FKM} for positively expansive maps.
For a hyperbolic metric not being self-similar only an inequality can be proved, see \cite{Fa}.
With this result we study the set of expanding factors.
We define the \emph{ideal expanding factor} as $\lideal=e^{\ent(f)/\dim(M)}$, where $\dim$ stands for topological dimension.
In Theorem \ref{thmIdeFactTrans}
we show that if $f\colon M\to M$ is an expansive axiom A diffeomorphism of a compact connected manifold 
with self-similar metric $\dist$ with ideal expanding factor on the non-wandering set then 
$f$ is a transitive Anosov diffeomorphism and the dimension of stable and unstable manifolds coincide.

In \S \ref{secHolonomy}
we show that our metric at small scales looks like a \emph{max norm} with respect to canonical coordinates. 
To explain the meaning of this statement let us recall that $f$ has \emph{canonical coordinates}\footnote{In the literature an expansive 
homeomorphism with canonical coordinates may be called \emph{topological Anosov} \cite{AH}, \emph{Smale space} \cite{Ruelle}, 
\emph{hyperbolic homeomorphism} \cite{Mane} and is equivalent to expansivity with the \emph{pseudo-orbit trancing property} 
or \emph{local product structure}.}
if for each $\epsilon>0$ there is $\delta>0$ such that 
$\dist(x,y)<\delta$ implies $W^s_\epsilon(x)\cap W^u_\epsilon(y)\neq\emptyset$. 
If in addition $2\epsilon$ is an expansive constant of $f$ then $W^s_\epsilon(x)\cap W^u_\epsilon(y)$ is a singleton 
and we can define a map $[\cdot,\cdot]$ by
\begin{equation}
\label{ecuCorchete}
 W^s_\epsilon(x)\cap W^u_\epsilon(y)=\{[x,y]\}
\end{equation}
whenever $\dist(x,y)<\delta$.
In Theorem \ref{teoCajaPIso} we show that for all $\epsilon>0$ there is $\delta>0$ 
such that if $0<\dist(x,y)<\delta$ 
then 
\[
 \left|\frac{\max\{\dist(x,[x,y]),\dist([x,y],y)\}}{\dist(x,y)}-1\right|<\epsilon.
\]
This result is related to \cite{Dov}*{Theorem 1.2'}.
In Theorem \ref{thmHoloIsoTrans}
we show that if $f$ is an expansive homeomorphism with canonical coordinates 
of a Peano continuum $M$ and holonomies are isometries then $f$ is transitive. 

In \S \ref{secIntMeas} we give an application to the ergodic theory of Anosov diffeomorphisms.
Using a self-similar hyperbolic metric we give a natural construction 
of the intrinsic measure (the probability measure with maximal entropy), 
also called \emph{Bowen-Margulis measure} \cite{HaKa}, of a topologically mixing 
expansive homeomorphism with canonical coordinates. 
This measure is obtained as a local product measure of Hausdorff measures on local stable and unstable sets 
with respect to a self-similar metric.
Analogous constructions can be found in \cites{Ham,Marg,Sinai}. 

I thank Mauricio Achigar, Federico Dalmao, Damián Ferraro, Ignacio Monteverde, Rafael Potrie, Armando Treibich and José Vieitez 
for useful conversations during the preparation of this work. 
I thank the referee for several corrections and suggestions.

\section{Expansivity and self-similarity}
\label{secExpSelfSim}
Let $(M,\rho)$ be a compact metric space and consider $f\colon M\to M$ a 
homeomorphism. 
We say that $f$ is \emph{expansive} if there is $\expc>0$ such that if $x\neq y$ then 
$\rho(f^n(x),f^n(y))>\expc$ for some $n\in\Z$. 

\begin{df}
\label{dfSelfSim}
Given a homeomorphism $f\colon M\to M$ we say that a compatible metric $\dist$ on $M$ 
is \emph{self-similar} if there are
constants $\expc>0,\lambda>1$ such that 
 if $\dist(p,q)\leq\expc$ then 
 \begin{equation}
  \label{ecuConforme}
    \max_{|i|=1}\dist(f^i(p),f^i(q))=\lambda\dist(p,q).
 \end{equation}
In this case we say that $\expc$ is an \emph{expansive constant} and $\lambda$ is the \emph{expanding factor} of the metric.
\end{df}

\subsection{Self-similar metrics for expansive homeomorphisms}

In this section we will construct a compatible self-similar metric for an arbitrary expansive homeomorphism 
of a compact metric space. 

\begin{rmk}
\label{rmkConstFathi}
 A self-similar metric could be obtained following the proof of \cite{Fa}*{Theorem 5.1}. 
 In doing so, we would have to change \cite{Fa}*{Equation (16)}\footnote{For reader's convenience we point that 
 Equation (16) is between Equations (5) and (7), 
 see \cite{Fa}*{p. 259}.} 
 with a $\sup$ in $n\in\Z$ instead of a bounded interval. 
 However, to simplify the proof of Theorem \ref{teoPerfHyp}
 we will start assuming Fathi's metric. 
\end{rmk}

As usual, the diameter of a set with respect to a metric $\rho$ is defined as 
\[
 \diam_{\rho}(A)=\sup_{x,y\in A}\rho(x,y)
\]
for all $A\subset M$.

\begin{thm}
\label{teoPerfHyp}
Every expansive homeomorphism $f\colon M\to M$ on a compact metric space 
admits a self-similar metric. 
\end{thm}

\begin{proof}
We start considering from \cite{Fa}*{Theorem 5.1} an adapted hyperbolic metric $\dist_F$ 
making $f$ and $f^{-1}$ Lipschitz. That is, 
there are $\expc_F>0$ and $k\geq\lambda>1$
such that
$$k\dist_F(x,y)\geq\max_{|i|=1}\dist_F(f^i(x),f^i(y))$$
for all $x,y\in M$ 
and 
\[
 \max_{|i|=1} \dist_F(f^i(x),f^i(y))\geq \lambda\dist_F(x,y)
\]
if $\dist_F(x,y)\leq\expc_F$. 
Consider the metric $\dist\colon M\times M\to\R$ defined as
\begin{equation}
 \label{ecuDefSS}
 \dist(x,y)=\max_{i\in\Z}\frac{\dist_F(f^i(x),f^i(y))}{\lambda^{|i|}}.
\end{equation}
Since $M$ is compact the distances are bounded, which implies that (\ref{ecuDefSS})
is a maximum and
$\dist$ is a metric.
Note that $\dist\geq\dist_F$.
To prove that the metrics $\dist_F$ and $\dist$ define the same topology on $M$ 
it only remains to show that for all $\epsilon>0$ there is $\delta>0$ such that if 
$\dist_F(x,y)<\delta$ then $\dist(x,y)<\epsilon$.
Consider two different points $x,y\in M$ 
and take an integer $j=j(x,y)\geq 0$ such that 
\begin{equation}
 \label{ecuCambioj}
 k^{j-1}<\frac{\diam_{\dist_F}(M)}{\dist_F(x,y)}\leq k^{j}.
\end{equation}
Since 
\[
 \dist_F(f^i(x),f^i(y))\leq \min\{k^{|i|}\dist_F(x,y),\diam_{\dist_F}(M)\}
\]
for all $i\in\Z$, we have that 
\begin{equation}
\label{ecuAcota}
\frac{\dist_F(f^i(x),f^i(y))}{\lambda^{|i|}}\leq \min\left\{\left(\frac{k}{\lambda}\right)^{|i|}\dist_F(x,y),\frac{\diam_{\dist_F}(M)}{\lambda^{|i|}}\right\}
\end{equation}
for all $i\in \Z$. 
Applying (\ref{ecuCambioj}), if $|i|\leq j-1$ then 
\[
 \left(\frac{k}{\lambda}\right)^{|i|}\dist_F(x,y)\leq\left(\frac{k}{\lambda}\right)^{j-1}\dist_F(x,y)\leq
 \frac{\diam_{\dist_F}(M)}{\lambda^{j-1}}.
\]
For $|i|\geq j$ it holds that $\diam_{\dist_F}(M)/\lambda^{|i|}\leq\diam_{\dist_F}(M)/\lambda^j$. 
Then
\[
 \min\{(k/\lambda)^{|i|}\dist_F(x,y),\diam_{\dist_F}(M)/\lambda^{|i|}\}\leq\diam_{\dist_F}(M)/\lambda^{j-1}
\]
for all $i\in\Z$.
Applying (\ref{ecuAcota}) we obtain
\begin{equation}
 \label{ecuHolderddF}
  \dist(x,y)=\max_{i\in\Z}\frac{\dist_F(f^i(x),f^i(y))}{\lambda^{|i|}}\leq \frac{k\diam_{\dist_F}(M)}{\lambda^{j-1}}.
\end{equation}
Since $j(x,y)\to+\infty$ as $\dist_F(x,y)\to 0$ the metrics $\dist$ and $\dist_F$ are compatible.

To prove (\ref{ecuConforme}) 
take $\expc>0$ such that if $\dist(x,y)<\expc$ then $\dist_F(x,y)<\expc_F$. 
We have that 
\[
\begin{array}{ll}
 \max\dist(f^{\pm 1}(x),f^{\pm 1}(y))& =
  \displaystyle\max_{i\in\Z}\frac{\dist_F(f^{i\pm 1}(x),f^{i\pm 1}(y))}{\lambda^{|i|}}\\
  & = \displaystyle\max_{i\in\Z}\frac{\dist_F(f^i(x),f^i(y))}{\lambda^{|i\pm 1|}}\\ 
\end{array}
\]
Since $\min\{|i+1|,|i-1|\}\geq|i|-1$ for all $i\in\Z$ we conclude
\[
 \max_{i\in\Z}\frac{\dist_F(f^i(x),f^i(y))}{\lambda^{|i\pm 1|}}
 \leq\max_{i\in\Z}\frac{\dist_F(f^i(x),f^i(y))}{\lambda^{|i|-1}}
\]
Notice that the difference between $\min |i\pm 1|$ and $|i|-1$, for an integer $i$, is only at $i=0$, 
where $\min |0\pm 1|=1$ and $|0|-1=-1$.
We know that if $\dist(x,y)<\expc$ then 
$\dist_F(x,y)<\expc_F$ and consequently
 \[
  \max \dist_F(f^{\pm1}(x),f^{\pm1}(y))\geq\lambda\dist_F(x,y).
 \]
Then
\[
 \max_{i\in\Z}\frac{\dist_F(f^i(x),f^i(y))}{\lambda^{|i\pm 1|}}\geq 
 \max_{|i|=1}\frac{\dist_F(f^i(x),f^i(y))}{\lambda^{|i\pm 1|}}\geq\lambda\dist_F(x,y).
\]
This proves that 
\[
 \max_{i\in\Z}\frac{\dist_F(f^i(x),f^i(y))}{\lambda^{|i\pm 1|}}
 =\max_{i\in\Z}\frac{\dist_F(f^i(x),f^i(y))}{\lambda^{|i|-1}}
\]
Finally 
\[
 \max_{i\in\Z}\frac{\dist_F(f^i(x),f^i(y))}{\lambda^{|i|-1}}=
 \lambda\max_{i\in\Z}\frac{\dist_F(f^i(x),f^i(y))}{\lambda^{|i|}}=\lambda\dist(x,y)
\]
which proves the result.
\end{proof}


\begin{rmk}
From Equation (\ref{ecuCambioj}) in the proof of Theorem \ref{teoPerfHyp}
we have that 
\[
 \left(\frac{\diam_{\dist_F}(M)}{\dist_F(x,y)}\right)^{\log_k(\lambda)}\leq \lambda^j.
\]
Applying (\ref{ecuHolderddF}) we conclude 
\[
 \dist_F(x,y)\leq\dist(x,y)\leq c[\dist_F(x,y)]^\alpha
\]
with $\alpha=\log_k(\lambda)\in(0,1)$ and $c>0$. 
That is, $\dist$ and $\dist_F$ are Hölder equivalent, as anticipated by Fried in \cite{Fried}*{Lemma 2}.
\end{rmk}

Let us give some examples with an explicit self-similar metric.

\begin{ex}[Shifts and subshifts]
\label{exSSMShift}
 Let $N^\Z$ be the space of sequences on $N$ symbols $\{0,1,2,\dots,N-1\}$. 
 For $a,b\colon\Z\to N$ define 
 $$T(a,b)=\max\{n\geq 0:a(i)=b(i)\text{ if }|i|\leq n\}.$$
 Given $\lambda>1$ define $\dist(a,b)=\lambda^{-T(a,b)}$. 
 It is easy to see that $\dist$ is self-similar with respect to the shift homeomorphism $\sigma\colon N^\Z\to N^\Z$ ($\sigma (a)_n=a_{n+1}$). 
 The expanding factor is $\lambda$. 
 If $X\subset N^\Z$ is a closed $\sigma$-invariant subset (a \emph{subshift}) 
 then the restricted metric is self-similar.
\end{ex}

\begin{ex}[Expansive homeomorphisms of surfaces]
\label{exSSSurfaces}
On compact surfaces we know that expansive homeomorphisms are conjugate to pseudo-Anosov diffeomorphisms, see \cites{Hi,L}. 
A pseudo-Anosov diffeomorphism $f\colon M\to M$ of a compact surface, by definition, 
has two invariant singular foliations with transverse measures $\mu_s,\mu_u$ that are expanded and contracted by a factor $\lambda>1$. 
To define a self-similar metric, for $p,q\in M$ consider the set $C^{su}(p,q)$ of 
curves $\alpha\colon[0,1]\to M$ from $p$ to $q$ such that 
there are $t=0=0< t_1<\dots<t_n=1$ such that each $\alpha_i=\alpha([t_i,t_{i+1}])$ 
is contained in a stable or an unstable leaf of the foliations. 
Let $l_s(\alpha)$ be the sum of the $\mu_s$-measures of the arcs $\alpha_i$ contained in a stable leaf.
Analogously, define $l_u(\alpha)$. 
Finally consider
\[
 \dist(p,q)=\inf_{\alpha\in C^{su}(p,q)} \max\{l_s(\alpha),l_u(\alpha)\}.
\]
It is easy to prove that $\dist$ is a compatible self-similar metric with expanding 
factor $\lambda$, where $\lambda$ is the expanding factor of the transverse measures.
\end{ex}

We remark that Theorem \ref{teoPerfHyp} can be applied to every expansive homeomorphism of a compact metric space, 
in particular to Anosov diffeomorphisms of compact smooth manifolds.
Under certain conditions a self-similar can be derived from a Riemannian metric.

\begin{ex}[Linear Anosov diffeomorphisms]
\label{exSomeAnosov}
Let $T\colon\R^n\to\R^n$ be a linear isomorphism inducing an Anosov automorphism $f$ of the torus $T^n$.
Assume that the stable subspace $E^s$ and the unstable subspace $E^u$ 
of $\R^n$ can be writen as $E^s=E^s_1\oplus\dots\oplus E^s_k$ and 
$E^u=E^u_1\oplus\dots\oplus E^u_l$ 
and there are real numbers $0<|a_1|,\dots,|a_k|<1$, $|b_1|,\dots,|b_l|>1$ such that 
$T(v)=a_iv$ for all $v\in E^s_i$ and 
$T(v)=b_jv$ for all $v\in E^u_j$.
Let $\|\cdot\|$ be a norm in $\R^n$.
Given $v\in\R^n$ consider $p\in E^s$ and $q\in E^u$ such that $v=p+q$ 
and take $p_i\in E^s_i$ and $q_j\in E^u_j$ such that 
$p=p_1+\dots+p_k$ and $q=q_1+\dots+q_l$. 
For 
$$\lambda=\min\{|a_i|^{-1},|b_j|:1\leq i\leq k,1\leq j\leq l\}$$ 
define 
$$\rho(v)=\max\{\|p_i\|^{\log(\lambda)\log|a_i|},\|q_j\|^{\log(\lambda)/\log|b_j|}:1\leq i\leq k,1\leq j\leq l\}.$$
The metric $\dist(p,q)=\rho(q-p)$ in $\R^n$ induces a self-similar metric on the torus with expanding factor $\lambda$.
\end{ex}


\subsection{Positively expansive maps}
In this brief section we indicate how to construct a self-similar metric 
for a positively expansive map.

\begin{thm}
\label{thmSSPosExp}
 If $f\colon M\to M$ is a positively expansive map of a compact metric space 
 then there is a compatible metric $\dist$ on $M$, $\expc>0$ and $\lambda>1$ such that 
 if $\dist(x,y)<\expc$ then $\dist(f(x),f(y))=\lambda\dist(x,y)$.
\end{thm}

\begin{proof}
By \cite{Reddy82} (see also \cite{Sakai03}) we know that there is a 
compatible metric $\dist_R$, $\expc>0$ and $\lambda>1$ 
such that if 
$\dist_R(x,y)<\expc$ then $\dist_R(f(x),f(y))\geq\lambda\dist_R(x,y)$. 
Also, this metric makes $f$ Lipschitz.
A self-similar metric can be defined by
\[
 \dist(x,y)=\max_{n\geq 0} \frac{\dist_R(f^n(x),f^n(y))}{\lambda^n}.
\]
The rest of the proof is analogous to the proof of Theorem \ref{teoPerfHyp}.
\end{proof}

\subsection{Basic properties of a self-similar metric}

In the following results we investigate 
simple but important properties of 
the self-similar metric. 

\begin{rmk}
 \label{rmkLips}
First note that Equation (\ref{ecuConforme}) implies that 
if $\dist(p,q)\leq\expc$ then 
\[
\begin{array}{l}
 \dist(f(p),f(q))\leq\lambda\dist(p,q)\text{ and }\\
 \dist(f^{-1}(p),f^{-1}(q))\leq\lambda\dist(p,q).
\end{array}
\]
These easily gives us that $f$ and $f^{-1}$ are Lipschitz. 
Moreover, considering $\dist'(p,q)=\min\{\dist(p,q),\expc\}$ we can assume that $\lambda$ itself is a Lipschitz constant for $f$ and $f^{-1}$.
In this case the expansivity constant should be reduced to $\expc'=\expc/\lambda$.
\end{rmk}

\begin{prop}
\label{propExpandeFut}
If $\dist$ is self-similar and $\dist(f(x),f(y))=\lambda\dist(x,y)$ 
then $$\dist(f^{k}(x),f^{k}(y))=\lambda^k\dist(x,y)$$ 
for all $k\geq 0$ such that $\lambda^{k-1}\dist(x,y)\leq\expc$.
\end{prop}

\begin{proof}
 For $k=0,1$ the result is trivial. 
 Consider $k\geq 2$ and assume that 
 $\dist(f^l(x),f^l(y))=\lambda^l\dist(x,y)$ for all $l=0,1,\dots,k-1$. 
 Therefore, 
 $$\dist(f^{k-1}(x),f^{k-1}(y))=\lambda\dist(f^{k-2}(x),f^{k-2}(y))$$ 
 and, in particular, $\lambda\dist(f^{k-1}(x),f^{k-1}(y))\neq\dist(f^{k-2}(x),f^{k-2}(y))$.
 Since 
 $$\lambda^{k-1}\dist(x,y)=\dist(f^{k-1}(x),f^{k-1}(y))\leq\expc$$ 
 we can apply
 Equation (\ref{ecuConforme}) to $p=f^{k-1}(x)$ and $q=f^{k-1}(y)$ to conclude $\dist(f^k(x),f^k(y))=\lambda \dist(f^{k-1}(x),f^{k-1}(y))$.
Since $$\dist(f^{k-1}(x),f^{k-1}(y))=\lambda^{k-1}\dist(x,y)$$ 
the proof ends.
 \end{proof}

\begin{prop}
\label{propExpandeFut2}
If $\dist$ is self-similar and $\lambda\dist(x,y)\neq\dist(f^{-1}(x),f^{-1}(y))$ 
then $$\dist(f^{k}(x),f^{k}(y))=\lambda^k\dist(x,y)$$ 
for all $k\geq 0$ such that $\lambda^{k-1}\dist(x,y)\leq\expc$.
\end{prop}

\begin{proof}
For $k=0$ there is nothing to prove. 
If $k=1$ we have $\dist(x,y)\leq\expc$. 
Since $\dist(x,y)\neq\lambda^{-1}\dist(f^{-1}(x),f^{-1}(y))$ 
we apply Equation (\ref{ecuConforme}) to obtain 
$\dist(f(x),f(y))=\lambda\dist(x,y)$. 
Then, applying Proposition \ref{propExpandeFut} the proof ends.
%
\end{proof}

\begin{prop}
\label{propEstContrae} 
If $y\in W^s_\expc(x)$
 then $\dist(f^n(x),f^n(y))=\lambda^{-n}\dist(x,y)$ 
 for all $n\geq 0$.
 Analogously, if $y\in W^u_\expc(x)$ 
 then $\dist(f^{-n}(x),f^{-n}(y))=\lambda^{-n}\dist(x,y)$ for all $n\geq 0$. 
\end{prop}

\begin{proof}
If $y\in W^s_\expc(x)$
and $\dist(f^n(x),f^n(y))\neq\lambda^{-n}\dist(x,y)$ 
for some $n\geq 0$ then 
there is $m\geq 0$ such that $\dist(f^{m+1}(x),f^{m+1}(y))\neq\lambda\dist(f^m(x),f^m(y))$. 
Applying Proposition \ref{propExpandeFut2} we contradict that $y\in W^s_\expc(x)$.
\end{proof}

\subsection{Dynamical triangles}

In what follows we will assume that $f\colon M\to M$ is an expansive homeomorphism with a self-similar metric $\dist$ on the compact space $M$. In addition, $\lambda>1$ and $\expc>0$ will denote the expanding factor and the expansivity constant respectively.
The purpose of this section is to prove Theorem \ref{teoDynTri} where we will give an estimate of the metric at small scales. 
This Theorem will be applied to prove that a self-similar metric is a \emph{max norm} at small scales assuming canonical coordinates, Theorem \ref{teoCajaPIso}.
However, the study of the present section does not assume canonical coordinates.

\begin{df}
We say that $(x,y)\in M\times M$ is a \emph{critical pair} if $0<\dist(x,y)<\expc$, 
$\dist(f^{-1}(x),f^{-1}(y))=\lambda\dist(x,y)$ and $\dist(f^2(x),f^2(y))=\lambda\dist(f(x),f(y))$. 
\end{df}
Note that if $(x,y)$ is a critical pair and $\dist(f(x),f(y))<\expc$ 
then $(f(x),f(y))$ is a critical pair for $f^{-1}$.
Fix a critical pair $(x,y)\in M\times M$ and define $c_n=\dist(f^n(x),f^n(y))$ for $n\in\Z$. 
By Proposition \ref{propExpandeFut} we have that if $n\geq 1$ then

\begin{eqnarray}
\label{ecuAsobreB1}
  c_{n+1}=\lambda^{n}c_1 & \text{ if } \lambda^{n-1}c_1\leq\expc,\\
\label{ecuAsobreB2}
  c_{-n}=\lambda^nc_0 & \text{ if } \lambda^{n-1}c_0\leq\expc.
\end{eqnarray}

If $c_1<\expc$, by Remark \ref{rmkLips} we have
\begin{equation}
 \label{ecuLipsCritv2}
 \frac 1{\lambda^2}\leq\frac{c_1}{\lambda c_0}\leq 1.
 \end{equation}

\begin{df}
 We say that $(x,y,z)$ is a \emph{dynamical triangle} 
 if $z=W^u_{\expc/2\lambda}(x)\cap W^s_{\expc/2\lambda}(y)$.
\end{df}

\begin{rmk}
\label{rmkNoObtuso}
For all $\epsilon>0$ there is $\delta>0$ such that if $(x,y,z)$ is a dynamical 
 triangle and $\dist(x,y)<\delta$ then $\max\{\dist(x,z),\dist(y,z)\}<\epsilon$. 
 The proof is as follows.
 If this is not the case we can take dynamical triangles $(x_n,y_n,z_n)$ such that 
 $\dist(x_n,y_n)\to 0$ but $\dist(x_n,z_n)$ is bounded away from zero (the other case is similar). 
 Then, two limit points of $x_n$ and $z_n$ contradict the expansivity of $f$.
\end{rmk}

Fix a dynamical triangle $(x,y,z)$ and define $a=\dist(x,z)$ and $b=\dist(z,y)$. 
Since $a,b\leq \expc/2\lambda$, there are $N_a,N_b>0$ such that
\begin{equation}
 \label{ecuNaNb}
 \begin{array}{l}
 \lambda^{N_a}a\leq\frac\expc 2<\lambda^{N_a+1}a,\\
 \lambda^{N_b}b\leq\frac\expc 2<\lambda^{N_b+1}b.  
 \end{array}
\end{equation}
Since $z\in W^u_\expc(x)$ and $z\in W^s_\expc(x)$, from Proposition \ref{propEstContrae} 
we have that 
\[
 \begin{array}{ll}
\dist(f^n(x),f^n(z))=\lambda^n a    & \text{ if } n \leq N_a,\\  
\dist(f^n(y),f^n(z))=\lambda^{-n} b & \text{ if } n \geq -N_b.
 \end{array}
\]
From this and the triangular inequality of the metric we conclude that
\begin{equation}
 \label{ecuCnTriang}
 |\lambda^n a-\lambda^{-n}b|\leq c_n\leq \lambda^na+\lambda^{-n}b
\end{equation}
whenever $-N_b\leq n\leq N_a$. 
By (\ref{ecuNaNb}) we can prove that $\lambda^{N_a} a>\lambda^{-N_a}b$
and $\lambda^{-N_b} a<\lambda^{N_b}b$.
Therefore,  (\ref{ecuCnTriang}) with $n=N_a$ and $n=-N_b$ implies

\begin{eqnarray}
  \label{ecuCnTriang21}
  \lambda^{N_a} a-\lambda^{-N_a}b\leq c_{N_a}\leq \lambda^{N_a}a+\lambda^{-N_a}b,\\
  \label{ecuCnTriang22}
  \lambda^{N_b} b-\lambda^{-N_b}a\leq c_{-N_b}\leq \lambda^{N_b}b+\lambda^{-N_b}a.
\end{eqnarray}
Thus
\begin{equation}
 \label{ecuCotaabv2}
 \frac{\lambda^{N_a} a-\lambda^{-{N_a}}b}{\lambda^{-{N_b}}a+\lambda^{N_b}b}
 \leq \frac{c_{N_a}}{c_{-{N_b}}}\leq \frac{\lambda^{N_a}a+\lambda^{-{N_a}}b}{\lambda^{N_b}b-\lambda^{-{N_b}} a}.
\end{equation}


\begin{df}
We say that a dynamical triangle $(x,y,z)$ is a \emph{critical triangle} 
if $(x,y)$ is a critical pair. 
\end{df}

\begin{prop}
 If $(x,y,z)$ is a critical triangle then
\begin{equation}
\label{ecuG1}
 \frac{a-\lambda^{-2N_a}b}{b+\lambda^{-2N_b}a}
 \leq \frac{c_1}{\lambda c_0}\leq 
 \frac{a+\lambda^{-2N_a}b}{b-\lambda^{-2N_b} a}. 
\end{equation}
\end{prop}

\begin{proof}
From (\ref{ecuNaNb}) and (\ref{ecuCnTriang}) 
we have that $c_n\leq \expc$ if $-N_b\leq n\leq N_a$.
This allows us to apply 
(\ref{ecuAsobreB1}) and (\ref{ecuAsobreB2}) to obtain 
$c_{N_a}=\lambda^{N_a-1} c_1$ and 
$c_{-N_b}=\lambda^{N_b} c_0$. 
Now the result follows from (\ref{ecuCotaabv2}).
\end{proof}
  
In the following lemmas we give some estimates that we need for the proof of Theorem \ref{teoDynTri}.
  
\begin{lem}
\label{lemCritUno}
For all $\epsilonb>0$ there is $\delta>0$ such that 
 for any critical triangle with
 $c_0<\delta$ it holds that
\begin{equation}
 \label{ecuCotaabv1}
  \frac 1{\lambda^2}-\epsilonb\leq\frac ab\leq 1+\epsilonb.
\end{equation}
\end{lem}

\begin{proof}
By (\ref{ecuLipsCritv2}) we know that $\frac{c_1}{\lambda c_0}\leq 1$.
From (\ref{ecuG1}) we have
$\frac{a-\lambda^{-2N_a}b}{b+\lambda^{-2N_b}a}\leq \frac{c_1}{\lambda c_0}$. 
Then
$\frac{a-\lambda^{-2N_a}b}{b+\lambda^{-2N_b}a}\leq 1.$ 
This implies that $\frac ab\leq\frac{1+\lambda^{-2N_a}}{1-\lambda^{-2N_b}}$. 
Considering again (\ref{ecuLipsCritv2}) and (\ref{ecuG1}) 
we have $\frac 1{\lambda^2}\leq\frac{c_1}{\lambda c_0}$ and 
$\frac{c_1}{\lambda c_0}\leq 
 \frac{a+\lambda^{-2N_a}b}{b-\lambda^{-2N_b} a}$. 
 This gives 
 \[
  \frac{\lambda^{-2}-\lambda^{-2N_a}}{1+\lambda^{-2}\lambda^{-2N_b}}\leq\frac ab\leq\frac{1+\lambda^{-2N_a}}{1-\lambda^{-2N_b}}
 \]
By Remark \ref{rmkNoObtuso}, if $c_0$ is small then $a$ and $b$ are small. Consequently $N_a$ and $N_b$ are large. 
This proves (\ref{ecuCotaabv1}).
\end{proof}

\begin{lem}
\label{lemCritDos}
There is $\epsilona>0$ such that for all $\epsilonb\in (0,\epsilona)$ 
there is $\delta>0$ such that for any critical triangle with $c_0<\delta$ it holds that

\begin{equation}
 \label{ecuBc02}
 \frac b{c_0}\leq 2.
\end{equation}
\end{lem}

\begin{proof}
For $\epsilonb>0$ given consider $\delta>0$ from Lemma \ref{lemCritUno}.
By (\ref{ecuCnTriang22}) and (\ref{ecuAsobreB2}) we have
$|c_0-b|\leq\lambda^{-2N_b}a$. From (\ref{ecuCotaabv1}) we conclude that
$
 \left|1-\frac b{c_0}\right|\leq\frac a{c_0}\lambda^{-2N_b}\leq\frac {b(1+\epsilonb)}{c_0}\lambda^{-2N_b}.
$
Then 
$
 \frac b{c_0}-1\leq \frac{b(1+\epsilonb)}{c_0}\lambda^{-2N_b}
$
and
$
 \frac{b}{c_0}\left(1-(1+\epsilonb)\lambda^{-2N_b}\right)\leq 1.
$
Finally, take $\epsilona>0$ such that $1-(1+\epsilona)\lambda^{-2N_b}\geq 1/2$.
\end{proof}

\begin{lem}
\label{lemCritTres}
For all $\epsilonb>0$ there is $\delta>0$ 
such that if $c_0=\dist(x,y)<\delta$ then 
\begin{equation}
 \label{ecuc0c1abv2}
 \left|\frac{c_1}{\lambda c_0}-\frac ab\right|<\epsilonb.
\end{equation}
\end{lem}

\begin{proof}
Using (\ref{ecuNaNb}) we have
\begin{equation}
\label{ecuG2}
\frac {\lambda^{-2N_a}b}a=\frac b{a\lambda^{N_a}\lambda^{N_a}}\leq \frac {2\lambda b}{\expc\lambda^{N_a}}\leq\frac \lambda{\lambda^{N_a+N_b}}.
\end{equation}
Analogously, 
\begin{equation}
\label{ecuG3}
 \frac{\lambda^{-2N_b}a}b\leq \frac \lambda{\lambda^{N_a+N_b}}.
\end{equation}
Then
(\ref{ecuG1}), (\ref{ecuG2}), (\ref{ecuG3}) and (\ref{ecuCotaabv1}) prove (\ref{ecuc0c1abv2}). 
\end{proof}

\begin{thm}
\label{teoDynTri}
 If $\dist$ is self-similar then for all $\epsilon>0$ there is $\delta>0$ such that 
 if $(x,y,z)$ is a dynamical triangle with $\diam(x,y,z)<\delta$ then 
 \[
  \left|\frac{\dist(x,y)}{\max\{\dist(x,z),\dist(z,y)\}}-1\right|<\epsilon.
 \]
\end{thm}

\begin{proof}
Consider $\epsilona>0$ from Lemma \ref{lemCritDos} and take $\epsilonc\in (0,\epsilona)$.
First we will show that there is $\delta>0$ such that if $(x,y,z)$ is critical and 
$c_0<\delta$ then
\begin{equation}
 \label{ecuDynTri}
   \left|\frac{c_0}{\max\{a,b\}}-1\right|\leq \epsilonc
\end{equation}
For $\epsilonb\in(0,\epsilona)$ consider $\delta>0$ satisfying Lemmas \ref{lemCritUno}, \ref{lemCritDos} and \ref{lemCritTres}.
From (\ref{ecuCotaabv1}) we have that $a\leq (1+\epsilonb)b$ and then
\begin{equation}
 \label{ecubmaxeps}
  b\leq\max\{a,b\}\leq (1+\epsilonb)b.
\end{equation}
Equation (\ref{ecuCnTriang21}) gives 
$
 |c_{N_a}-\lambda^{N_a}a|\leq\lambda^{-N_a}b.
$
From (\ref{ecuAsobreB1}) we know that $c_{N_a}=\lambda^{N_a-1}c_1$. Therefore
$
 \left|\frac{c_1}{\lambda c_0}-\frac a{c_0}\right|\leq\frac b{c_0}\lambda^{-2N_a},
$
which jointly with (\ref{ecuc0c1abv2}) implies
$
 \left|\frac ab-\frac a{c_0}\right|\leq\epsilonb+\frac b{c_0}\lambda^{-2N_a}
$
and 
$
 \left|1-\frac b{c_0}\right|\leq\frac ba\left[\epsilonb+\frac b{c_0}\lambda^{-2N_a}\right].
$
By (\ref{ecuCotaabv1}) we have that $\frac ba\leq\frac{\lambda^2}{1-\lambda^2\epsilonb}$. 
Then
\begin{equation}
 \label{ecubsobrec0}
 \left|1-\frac b{c_0}\right|\leq \frac{\lambda^2}{1-\lambda^2\epsilonb}\left[\epsilonb+\frac b{c_0}\lambda^{-2N_a}\right].
\end{equation}
This inequality and (\ref{ecuBc02}) gives us
$
 \left|1-\frac b{c_0}\right|\leq \frac{\lambda^2}{1-\lambda^2\epsilonb}(\epsilonb+2\lambda^{-2N_a}).
$
By (\ref{ecubmaxeps}) we have that 
$
 \left|\frac b{\max\{a,b\}}-1\right|\leq\frac\epsilonb{1+\epsilonb}.
$
Adding the last two inequalities and applying the triangular inequality we have
$
  \left|\frac b{\max\{a,b\}}-\frac b{c_0}\right|\leq \frac{\lambda^2}{1-\lambda^2\epsilonb}(\epsilonb+2\lambda^{-2N_a})+\frac\epsilonb{1+\epsilonb}.
$
That is
$
  \left|\frac{c_0}{\max\{a,b\}}-1\right|\leq \frac{c_0}b\left[\frac{\lambda^2}{1-\lambda^2\epsilonb}(\epsilonb+2\lambda^{-2N_a})+\frac\epsilonb{1+\epsilonb}\right].
$
From (\ref{ecubmaxeps}) we have that 
\[
 \frac{\max\{a,b\}}{1+\epsilonb}\leq b
\]
and then 
\begin{equation}
 \label{ecuC0b}
  \frac{c_0}b\leq\frac{c_0}{\max\{a,b\}}(1+\epsilonb)\leq 2(1+\epsilonb)
\end{equation}
because $c_0\leq2\max\{a,b\}$ (triangular inequality).
Then
\[
  \left|\frac{c_0}{\max\{a,b\}}-1\right|\leq 
  2(1+\epsilonb)\left[\frac{\lambda^2}{1-\lambda^2\epsilonb}(\epsilonb+2\lambda^{-2N_a})+\frac{\epsilonb}{1+\epsilonb}\right].
\]
From Remark \ref{rmkNoObtuso} we have that if $\delta$ is small then $\lambda^{-2N_a}$ is small.
Then, it is clear that if $\epsilonb$ is sufficiently small then we obtain (\ref{ecuDynTri}).

For an arbitrary dynamical triangle $(x',y',z')$ the proof is reduced to the case of a critical triangle as follows. 
\azul{Consider $\delta>0$ as before (satisfying the previous lemmas). 
Assume that $\diam(x',y',z')<\delta$.
Suppose that for some $n>0$ the pair $(f^n(x'),f^n(y'))$ is critical. 
If we define $x=f^n(x')$, $y=f^n(y')$ and $z=f^n(z')$ then 
\begin{eqnarray*}
\dist(x',y')=\lambda^n c_0, \\
\dist(y',z')=\lambda^nb,\\
\dist(x',z')=\lambda^{-n}a.
\end{eqnarray*}
In particular, $c_0<\delta$.}
By (\ref{ecuCotaabv1}) we have that $\frac{a}{b}<1+\epsilonb<\lambda^2$ (assuming that $\epsilonb<\lambda^2-1$). 
Then $\max\{\lambda^{-2n}a,b\}=b$ and
\[
 \frac{\dist(x',y')}{\max\{\dist(x',z'),\dist(z',y')\}}=\frac{c_0}{\max\{\lambda^{-2n}a,b\}}=\frac{c_0}b.
\]
From (\ref{ecubmaxeps}) we have that 
\[
 \frac 1b-\frac 1{\max\{a,b\}}\leq \frac{\epsilonb}{b(1+\epsilonb)}.
\]
Since for the critical triangle $(x,y,z)$ we have proved (\ref{ecuDynTri}), we have
\[
 \left|\frac{c_0}{b}-1 \right|\leq 
 \left|\frac{c_0}{b}-\frac{c_0}{\max\{a,b\}} \right|+\left|\frac{c_0}{\max\{a,b\}}-1 \right|
 \leq \frac{c_0\epsilonb}{b(1+\epsilonb)}+\epsilonc.
\]
Jointly with (\ref{ecuC0b}) we obtain
\[
 \left|\frac{c_0}{b}-1 \right|\leq 
 2\epsilonb+\epsilonc.
\]
Finally, given $\epsilon>0$ it is sufficient to take $\epsilonb=\epsilonc=\epsilon/3$.
\end{proof}

\begin{rmk}
 In \cite{Ruelle} Ruelle asked if it can be found a hyperbolic metric such that
\begin{equation}
 \label{ecuSS3}
 \begin{array}{l}
\dist([x,y],x)\leq L \dist(x,y),\\
\dist([x,y],y)\leq L \dist(x,y).   
 \end{array}
\end{equation}
for some $L\geq 1$, assuming expansivity and canonical coordinates. 
Such property of a hyperbolic metric was obtained by Fried in \cite{Fried83} (see also \cite{Sakai2001}). 
From Thereom \ref{teoDynTri}, taking $z=[x,y]$,
we have that: if $f\colon M\to M$ is expansive, has canonical coordinates and $\dist$ is self-similar 
 then for all $\epsilon>0$ 
 there is $\delta>0$ such that if 
 $\dist(x,y)<\delta$ then 
\[
 \begin{array}{l}
\dist([x,y],x)\leq (1+\epsilon) \dist(x,y),\\
\dist([x,y],y)\leq (1+\epsilon) \dist(x,y).   
 \end{array}
\]
This property was previously proved by Dovbish in \cite{Dov}*{Remark 1.5}.
\end{rmk}

\section{Topological entropy}
\label{secTopEnt}
Let $(M,\dist)$ be a compact metric space. 
For $\epsilon>0$ we say that $$\U=\{A_1,\dots,A_n\}$$ is an $(\epsilon,\dist)$-\emph{cover} 
of $X\subset M$ if $\cup_{i=1}^nA_i=X$ and $\diam_{\dist}(A_i)<\epsilon$ for all $i=1,\dots,n$. 
Define 
$$\cov_\epsilon(X,\dist) =\min\{\card(\U): \U \text{ is an } (\epsilon,\dist)\text{-cover of }X\}$$
and 
\[
 \dist^f_n(x,y)=\max_{|k|\leq n} \dist(f^k(x),f^k(y)).
\]

\begin{rmk}
 Note 
 that if $f$ is expansive and $\dist$ is a self-similar metric with expansive constant $\expc$ 
 and expanding factor $\lambda$ then 
 \begin{equation}
  \label{ecuDists}
  \dist^f_n(x,y)=\lambda^n\dist(x,y)
 \end{equation}
 if $\dist(x,y)\leq\expc/\lambda^n$. 
 Equation (\ref{ecuDists}) also holds if $\dist^f_n(x,y)\leq\expc$.
\end{rmk}

\begin{prop}
\label{propSepHypMet}
 If $f$ is expansive and $\dist$ is a self-similar metric with expansive constant $\expc$ 
 and expanding factor $\lambda$ then 
 \[
  \cov_{\expc/\lambda^k}(X,\dist)=\cov_\expc(X,\dist_k^f)
 \]
for all $k\geq 0$ and $X\subset M$.
\end{prop}

\begin{proof}
If $\U$ is an $(\expc/\lambda^k,\dist)$-cover of $X$
then $\diam_{\dist}(U)<\expc/\lambda^k$ for all $U\in\U$. 
Equation (\ref{ecuDists}) implies that 
$\diam_{\dist^f_k}(U)<\expc$ and 
then $\U$ is an $(\expc,\dist_k^f)$-cover of $X$. 
Then $\cov_{\expc/\lambda^k}(X,\dist)\geq\cov_\expc(X,\dist_k^f)$. 

To prove the converse inequality consider $\U$ an $(\expc,\dist_k^f)$-cover of $X$. 
Given $U\in\U$ we have that $\diam_{\dist^f_k}(U)<\expc$, that is, 
$\diam(f^j(U))<\expc$ if $|j|\leq k$.
If $\diam(U)\geq\expc/\lambda^k$ then $\max\{\diam(f^k(U)),\diam(f^{-k}(U))\}\geq\expc$ (which would be a contradiction). 
Then $\U$ is an $(\expc/\lambda^k,\dist)$-cover. This proves the other inequality.
\end{proof}

It is known that 
the limit
 \[
  \ent(X)=\lim_{n\to+\infty} \frac 1n\log(\cov_\expc(X,\dist_n^f))
 \]
exists, is finite and independent of the expansive constant $\expc$ in the above expression, see for example \cite{BS}.
This limit is the \emph{topological entropy} of $X$ associated to $f$.
The \emph{topological entropy} of $f$ is defined as $\ent(f)=\ent(M)$.


Define 
\begin{eqnarray*}
 \dist^+_n(x,y)=\max_{0\leq k\leq n} \dist(f^k(x),f^k(y)),\\
 \dist^-_n(x,y)=\max_{0\leq k\leq n} \dist(f^{-k}(x),f^{-k}(y)).
\end{eqnarray*}
Given $X\subset M$ define
\begin{eqnarray*}
 \ent^+(X)=\lim_{n\to+\infty} \frac 1n\log(\cov_\expc(X,\dist_n^+)),\\
 \ent^-(X)=\lim_{n\to+\infty} \frac 1n\log(\cov_\expc(X,\dist_n^-)).
\end{eqnarray*}
These numbers depend on $f$.
\begin{rmk}
\label{obsDefEnt}
The topological entropy considered in \cite{Fa} (and in most texts, as for example in \cite{BS}) is $\entf(f)=\ent^+(M)$. 
It is easy to prove that $$2\ent^+(M)=\ent(f).$$ 
To avoid the introduction of a factor 2 in Equation (\ref{ecuCapEntLogL}) we defined the entropy as we did (which, again, 
is not the standard way).
\end{rmk}

\begin{prop}
\label{propDefEnt}
We have that
$$\ent^+(M)=\ent^-(M)=\ent(M)/2.$$ 
\end{prop}

\begin{proof}
Notice that $\U$ is an $(\expc, \dist^+_n)$-cover of $M$ if and only if 
$f^n(\U)$ is an $(\expc, \dist^-_n)$-cover of $M$. 
This proves that $\ent^+(M)=\ent^-(M)$. 

Note that if $\U$ is an $(\expc,\dist_n)$ cover of $M$ then 
$f^{-n}(\U)$ is an $(\expc,\dist^+_{2n-1})$ cover.
This implies that $\cov_\expc(M,\dist^+_{2n-1})\leq\cov_\expc(M,\dist_n)$ 
and $2\ent^+(M)\leq \ent(M)$.
The converse inequality is analogous.
\end{proof}

We say that a set $X\subset M$ is \emph{stable} if $\diam(f^n(X))\to 0$ as $n\to+\infty$. 
We say that it is \emph{unstable} if $\diam(f^n(X))\to 0$ as $n\to-\infty$.
If $X$ is unstable and $\diam(f^{-n}(X))\leq\expc$ for all $n\geq 0$ then 
\begin{equation}
 \label{ecuCovMas}
 \cov_{\expc/\lambda^k}(X,\dist)=\cov_\expc(X,\dist^+_k) 
\end{equation}
if the metric is self-similar.

\subsection{Canonical coordinates}

Let $f\colon M\to M$ be an expansive homeomorphism with canonical coordinates of a compact metric space.
Recall the map $[,]$ defined in Equation (\ref{ecuCorchete}) 
giving the intersection of local stable and unstable sets.
For $p\in M$ and $r>0$ small define 
$$C_r(p)=[W^s_r(p),W^u_r(p)]=\{[x,y]:x\in W^s_r(p),y\in W^u_r(p)\}.$$
We will say that $C_r(p)$ is a \emph{product box} around $p$. 
The sets $[\{x\},W^u_r(p)]$ and 
$[W^s_r(p), \{y\}]$ will be called \emph{plaques} of the product box.

We recall that the spectral decomposition theorem states that for an 
expansive homeomorphism $f$ with canonical coordinates the non-wandering set $\Omega(f)$ 
is a disjoint union $B_1\cup\dots\cup B_l$ of compact $f$-invariant sets 
and each $f\colon B_i\to B_i$ is transitive. 
The sets $B_i$ are called \emph{basic sets}.
The spectral decomposition for hyperbolic diffeomorphisms of smooth manifolds is due to Smale.
A proof in the topological setting can be found in \cite{AH}*{Theorem 3.4.4}.

Assume that $3\expc$ is an expansive constant and $\dist$ is self-similar 
with expanding factor $\lambda$.

\begin{prop}
\label{propCapCajaPlaca}
If $r<\expc$ and $C=C_r(p)$ is a product box then
 \[
  \ent^+(C)=\ent^+(P^u)\text{ and }\ent^-(C)=\ent^-(P^s)
 \]
for every unstable plaque $P^u$ and every stable plaque $P^s$ of $C$.
\end{prop}

\begin{proof}
 Let us only prove the first equality. 
 Since $P^u\subset C$ we have that $\ent^+(P^u)\leq \ent^+(C)$. 
 Let $\pi\colon C\to P^u$ be the canonical projection. 
 Consider $p,q\in C$. 
 By the triangular inequality we have
 \[
 \begin{array}{ll}
  \dist(f^n(p),f^n(q))\leq &\dist(f^n(p),f^n(\pi(p)))+\\
  & \dist(f^n(\pi(p)),f^n(\pi(q)))+\dist(f^n(\pi(q)),f^n(q))
 \end{array}
 \]
for all $n\in\Z$. If $n\geq 0$ then 
\[
 \begin{array}{ll}
 \dist(f^n(p),f^n(q))\leq &\lambda^{-n}\dist(p,\pi(p))+\\
  & \dist(f^n(\pi(p)),f^n(\pi(q)))+\lambda^{-n}\dist(\pi(q),q)\\
 \end{array}
 \]
 Thus
 \[
\dist(f^n(p),f^n(q))\leq \dist(f^n(\pi(p)),f^n(\pi(q)))+2\expc/\lambda^{n}.  
 \]
Therefore, if $\dist(f^k(\pi(p)),f^k(\pi(q)))\leq\expc$ for all $k=0,1,\dots,n$ 
then $$\dist(f^k(p),f^k(q))\leq3\expc$$ for all $k=0,1,\dots,n$. 
That is, if $\dist^+_n(\pi(p),\pi(q))\leq\expc$ then $\dist^+_n(p,q)\leq3\expc$. 
Consequently, if $\U$ is an $(\expc,\dist^+_n)$-cover of $P^u$ then 
$\pi^{-1}(\U)$ is a $(3\expc,\dist^+_n)$-cover of $C$.
This implies that $\cov_\expc(P^u,\dist^+_n)\geq\cov_{3\expc}(C,\dist^+_n)$ 
and $\ent^+(P^u)\geq \ent^+(C)$.
\end{proof}

\begin{cor}
\label{corEntCteLoc}
 If $P^u_1$ and $P^u_2$ are unstable plaques of a common product box $C$ then 
 $\ent^+(P^u_1,\dist)=\ent^+(P^u_2,\dist)$. 
 Analogous for stable plaques.
\end{cor}

\begin{proof}
By Proposition \ref{propCapCajaPlaca} we know that $\ent^+(P^u_i)=\ent^+(C)$.
\end{proof}

\subsection{Homogeneous entropy}

We say that $f$ has \emph{homogeneous entropy} if 
for all $x,y\in M$ it holds that 
$\ent(W^u_\expc(x))=\ent(W^u_\expc(y))$ 
and 
$\ent(W^s_\expc(x))=\ent(W^u_\expc(y))$.

\begin{prop}
If $f$ has homogeneous entropy then 
 $\ent(W^u_\expc(x))$ and $\ent(W^s_\expc(x))$ 
 do not depend on the expansivity constant $\expc$.
\end{prop}

\begin{proof}
It is direct from the definitions noting that $\ent(X)=\ent(f(X))$ for all $X\subset M$.
\end{proof}

\begin{prop}
\label{propHomoEntEvenSp}
If $f$ is expansive with canonical coordinates and homogeneous entropy then
 \[
  \ent^+(W^u_\expc(x))=\ent^-(W^s_\expc(x))=\frac 12\ent(M)
 \]
for all $x\in M$. 
\end{prop}

\begin{proof}
We will prove that $\ent^+(W^u_\epsilon(x))=\frac 12\ent(M)$. 
Let $C_1,\dots,C_p$ be a cover of $M$ by product boxes. 
 By Proposition \ref{propDefEnt} we have that 
 $\frac12\ent(M)=\ent^+(M)$.
 As in \cite{BS}*{Proposition 2.5.5} we can prove that
 \[
  \ent^+(M)=\max_{i=1,\dots,p} \ent^+(C_i).
 \]
Suppose that $\ent^+(M)=\ent^+(C_j)$ for some $j$. 
If $P^u$ is an unstable plaque of $C_j$, by Proposition \ref{propCapCajaPlaca}
we know that 
$\ent^+(C_j)=\ent^+(P^u)$. This finishes the proof.
\end{proof}

We say that a basic set $\Lambda\subset \Omega(f)$ is \emph{extremal} if it is an attractor or a repeller 
of the spectral decomposition of the non-wandering set.

\begin{prop}
An expansive homeomorphism with canonical coordinates 
has homogeneous entropy 
if and only if 
$\ent(M)=\ent(\Lambda)$ for every extremal basic set $\Lambda\subset\Omega(f)$.
\end{prop}

\begin{proof}
To prove the direct part
suppose that $\Lambda$ is an attractor. 
In this case we have that $W^u_\expc(x)\subset\Lambda$ for all $x\in\Lambda$. 
If we apply Proposition \ref{propHomoEntEvenSp} to $f\colon M\to M$ and $f\colon\Lambda\to\Lambda$ 
we obtain 
\[
\azul{\ent(M)=2\ent^+(W^u_\expc(x))=\ent(\Lambda)}
\]
for all $x\in\Lambda$. 

To prove the converse assume that $\ent(M)=\ent(\Lambda)$ for every extremal basic set $\Lambda\subset\Omega(f)$.
Since $\ent(\Lambda')\leq\ent(M)$ for every basic set $\Lambda'$ we have that $\ent(\Lambda')\leq\ent(\Lambda)$ for every extremal $\Lambda$.
Given $x\in M$ there are $y\in B_\expc(x)$ and an attractor $\Lambda_a$ such that 
$\dist(f^n(x),\Lambda_a)\to 0$ as $n\to+\infty$.
Applying Corollary \ref{corEntCteLoc} we conclude that $\ent^+(W^u_\expc(x))=\ent^+(W^u_\expc(z))$ for $z\in\Lambda_a$.
\end{proof}

\begin{thm}
\label{thmCapMitad}
 If $f\colon M\to M$ is a transitive expansive homeomorphism with canonical coordinates
 of a compact metric space $M$ 
 then $f$ has homogeneous entropy.
\end{thm}

\begin{proof}
 Take $p,q\in M$ and
 consider $\epsilon_p,\epsilon_q>0$. 
 Let $C_p,C_q$ be boxes around $p,q$ respectively such that 
 $W^s_{\epsilon_p}(p)$ is a plaque of $C_p$ and
 $W^s_{\epsilon_q}(q)$ is a plaque of $C_q$. 
 Take $x\in M$ with orbit dense in $M$. 
 Assume that $x\in C_p$ and denote by $P_x$ the stable plaque of $x$ in $C_p$. 
 Take $n\geq 0$ such that $f^n(P_x)\subset C_q$. 

 We have that 
 $\ent^-(P_x)=\ent^-(f^n(P_x))$. 
 Let $Q_x$ be the stable plaque of $f^n(x)$ in $C_q$. 
 Since $f^n(P_x)\subset Q_x$ we have that 
 $\ent^-(f^n(P_x))\leq\ent^-(Q_x)$. 
 By Corollary \ref{corEntCteLoc} we have that $\ent^-(Q_x)=\ent^-(W^s_{\epsilon_q}(q))$. 
 Then, we have proved that $$\ent^-(W^s_{\epsilon_p}(p))\leq\ent^-(W^s_{\epsilon_q}(q)).$$ 
 Analogously, we can prove the converse inequality and the proof ends. 
\end{proof}

\subsection{Capacity}
\label{secCap}
Given $X\subset M$ the \emph{capacity}
is defined as 
\begin{equation}
 \label{ecuCapDf}
  \ucap(X,\dist)=\lim_{\epsilon\to 0} \frac{\log(\cov_\epsilon(X,\dist))}{-\log(\epsilon)}
\end{equation}
whenever this limit exists.
The following result is based on \cite{Fa}*{Theorem 5.3} where an inequality 
is proved.

\begin{thm}
\label{teoEntCap}
 If $f$ is expansive and $\dist$ is a self-similar metric with expanding factor $\lambda$ then 
 the limit (\ref{ecuCapDf}) exists and
 \begin{equation}
  \label{ecuCapEntLogL}
  \ucap(X,\dist)=\frac{\ent(X)}{\log(\lambda)}
 \end{equation}
 for every $X\subseteq M$.
\end{thm}

\begin{proof}
Given $\epsilon>0$ small consider a positive integer  $n(\epsilon)$ such that
$$\expc/\lambda^{n(\epsilon)+1}\leq\epsilon<\expc/\lambda^{n(\epsilon)}.$$
Then 
\[
\cov_{\expc/\lambda^{n(\epsilon)+1}}(X,\dist)\geq
\cov_\epsilon(X,\dist)\geq
\cov_{\expc/\lambda^{n(\epsilon)}}(X,\dist)
\]
and 
\[
\frac{\log(\cov_{\expc/\lambda^{n(\epsilon)+1}}(X,\dist))}
{-\log(\expc/\lambda^{n(\epsilon)})}
\geq
\frac{\log(\cov_\epsilon(X,\dist))}
{-\log(\epsilon)}
\geq
\frac{\log(\cov_{\expc/\lambda^{n(\epsilon)}}(X,\dist))}
{-\log(\expc/\lambda^{n(\epsilon)+1})}.
\]
By Proposition \ref{propSepHypMet} we have that $\cov_{\expc/\lambda^k}(X,\dist)=\cov_\expc(X,\dist_{k}^f)$. 
Then
\[
\frac{\log(\cov_{\expc}(X,\dist^f_{n(\epsilon)+1}))}
{n(\epsilon)\log\lambda-\log(\expc)}
\geq
\frac{\log(\cov_\epsilon(X,\dist))}
{-\log(\epsilon)}
\geq
\frac{\log(\cov_{\expc}(X,\dist^f_{n(\epsilon)}))}
{(n(\epsilon)+1)\log\lambda-\log(\expc)}
\]
Since $\lim_{\epsilon\to 0} n(\epsilon)=+\infty$
we conclude that
$$\frac{\ent(X)}{\log(\lambda)}\geq\ucap(X,\dist)\geq\frac{\ent(X)}{\log(\lambda)}$$
which proves the result.
\end{proof}

\begin{prop}
\label{propCapsu}
Assume that $f\colon M\to M$ is an expansive homeomorphism with self-similar metric $\dist$ and expanding factor $\lambda$. 
For every unstable set $X$ and every stable set $Y$ it holds that:
\[
 \ucap(X,\dist)=\frac{\ent^+(X)}{\log(\lambda)}=\frac{\ent(X)}{\log(\lambda)}
\]
and 
\[
 \ucap(Y,\dist)=\frac{\ent^-(Y)}{\log(\lambda)}=\frac{\ent(Y)}{\log(\lambda)}.
\]
\end{prop}

\begin{proof}
 It is analogous to the proof of Theorem \ref{teoEntCap} using Equation (\ref{ecuCovMas}).
\end{proof}

In this case it is easy to see that 
$\ent^-(X)=\ent^+(Y)=0$.

\begin{cor}
\label{corCapCteLoc}
 If $P^u_1$ and $P^u_2$ are unstable plaques of a common product box $C$ then 
 $\ucap(P^u_1,\dist)=\ucap(P^u_2,\dist)$. 
 Analogous for stable plaques.
\end{cor}

\begin{proof}
It follows by Corollary \ref{corEntCteLoc}
and Proposition \ref{propCapsu}.
\end{proof}

\subsection{Expanding factors}
The purpose of this section is to study the set of expanding factors $\lambda$ 
of a self-similar metric 
that can be associated to a given expansive homeomorphism $f\colon M\to M$. 
Note that if there is a self-similar metric $\dist$ with expanding factor $\lambda$ 
then the metric $\dist'(x,y)=[\dist(x,y)]^\alpha$, $0<\alpha<1$, is  
self-similar with expanding factor $\lambda^\alpha<\lambda$. 
This means that the set of expanding factors is an interval. 

Let $\lambda_{\sup}(f)\in(1,+\infty]$ be such that 
every $\lambda<\lambda_{\sup}$ is the expanding factor of a self-similar metric for $f$, 
and every $\lambda>\lambda_{\sup}$ is not.

\begin{rmk}
The metric defined in Example \ref{exSomeAnosov} 
has expanding factor $\lambda_{\sup}$.
\end{rmk}

The following result characterizes the systems with $\lambda_{\sup}$ infinite.
The topological dimension \cite{HW} of $M$ will be denote as $\dim(M)$.
In \cite{HW} it is shown that $\dim(M)\leq\ucap(M,\dist)$ for every metric $\dist$. 

\begin{prop}
 $\lambda_{\sup}=+\infty$ if and only if $M$ is totally disconnected.
\end{prop}

\begin{proof}
 Assume that $M$ is totally disconnected. 
 From \cite{KR}*{Corollary 2.9} we known that $f$ is conjugate with a subshift. 
 Recall that in Example \ref{exSSMShift} we gave a self-similar metric for 
 the shift homeomorphism with an arbitrary expanding factor $\lambda>1$.
 
 If $M$ is not totally disconnected then the topological dimension of $M$ 
 is positive.
Applying Theorem \ref{teoEntCap} we have
\begin{equation}
\label{eqDimTop}
\dim(M)\log(\lambda)\leq \ent(f)
\end{equation}
for every expanding factor $\lambda$ of a hyperbolic self-similar metric.
This implies 
$$\lambda\leq e^{\ent(f)/\dim(M)}$$
and $\lambda_{\sup}$ is finite.
\end{proof}

\begin{rmk}
 As noticed in \cite{Fa}, Equation (\ref{eqDimTop}) implies that: 1) if a compact metric space admits an expansive homeomorphism then 
its topological dimension is finite (a result first proved by Mañé \cite{Ma}) and 2) if $\dim(M)>0$ then 
every expansive homeomorphism of $M$ has positive topological entropy.
\end{rmk}

In what follows assume that $\dim(M)>0$.
Define the \emph{ideal expanding factor} of $f$ as
$$\lideal=e^{\ent(f)/\dim(M)}.$$
Obviously, if $\lambda$ is an expanding factor then $\lambda\leq\lideal$.

\begin{ex}[Pseudo-Anosov maps again]
On surfaces every expansive homeomorphism admits a self-similar metric with ideal expanding factor.
This is because the metric that we constructed in Example \ref{exSSSurfaces} 
expands with the factor associated to the transverse measures. 
Since the topological entropy of a pseudo-Anosov map is $2\log(\lambda)$ 
and the topological dimension of a surface is $2$, the stretching factor of a pseudo-Anosov 
diffeomorphism is $\lideal$. 
\end{ex}

%
%

It is of interest for the next result to remark that an expansive axiom A diffeomorphism may not be Anosov, 
for example there are quasi-Anosov diffeomorphisms that are not Anosov \cite{FR}. 
Quasi-Anosov diffeomorphisms are known to be axiom A and expansive.

\begin{thm}
\label{thmIdeFactTrans}
Let $f\colon M\to M$ be an expansive homeomorphism of a compact connected manifold 
with self-similar metric $\dist$. 
If $f\colon\Omega(f)\to\Omega(f)$ has ideal expanding factor then $\Omega(f)$ has non-empty interior.
If in addition $f$ is an axiom A diffeomorphism then 
$f$ is a transitive Anosov diffeomorphism and the dimension of stable and unstable manifolds coincide.
\end{thm}

\begin{proof}
It is known \cite{Bo70} (see also \cite{HaKa}*{Equation (3.3.1)}) that the topological entropy of 
$f$ restricted to the non-wandering set equals the topological entropy 
of $f$ in $M$. 
Since $f$ has ideal expanding factor on $\Omega(f)$ we can apply Theorem \ref{teoEntCap} 
to conclude that $\dim(\Omega(f))=\dim(M)$.
From \cite{HW}*{Theorem IV 3} we know that $\Omega(f)$ has non-empty interior. 

If $f$ is axiom A we can apply \cite{Fisher} to conclude that 
$\Omega(f)=M$ and that 
$f$ is a transitive Anosov.
From Theorem \ref{thmCapMitad}, $f$ has homogeneous entropy. 
Applying Proposition \ref{propHomoEntEvenSp}
we have that
 \[
  \ent^+(W^u_\expc(x))=\ent^-(W^s_\expc(x))=\frac 12\ent(M)
 \]
for all $x\in M$. 
\azul{By Proposition \ref{propCapsu}} we conclude that
$$\ucap(W^u_\expc(x),\dist)=\ucap(W^s_\expc(x),\dist)=\frac{\dim(M)}2.$$ 
If $\dim(W^s(x))\neq\dim(W^u(x))$ then 
one of these numbers is strictly greater than $\dim(M)/2$. 
We arrive to a contradiction because the capacity is greater or equal than the dimension.
\end{proof}

As a consequence, an Anosov diffeomorphism of a three dimensional manifold 
does not admit a self-similar metric with ideal expanding factor.

\section{Holonomy on canonical coordinates}
\label{secHolonomy}

In this section we give some simple properties of pseudo isometries and the Hausdorff measure. 
Assuming the existence of canonical coordinates we show that holonomies are pseudo isometries, 
and consequently the Hausdorff measure is preserved by holonomy.
We also show that on a Peano continuum the condition of isometric holonomies 
implies the transitivity of the homeomorphism.

\subsection{Pseudo isometries and Hausdorff measure}

Given metric spaces $(X_i,\dist_i)$, $i=1,2$ we say that 
a homeomorphism $h\colon X_1\to X_2$ is a \emph{pseudo isometry} 
if for all $\epsilon>0$ there is $\delta>0$ such that 
if $0<\dist_1(x,y)<\delta$ then 
\[
 \left|\frac{\dist_2(h(x),h(y))}{\dist_1(x,y)}-1\right|<\epsilon.
\]
In this case, $X_1$ and $X_2$ are said to be \emph{pseudo isometric}.

Given a compact metric space $(X,\dist)$ and $r>0$
define $C_r(X,\dist)$ as the set of countable covers $\U$ of $X$ 
such that $\diam(U)<r$ for all $U\in \U$.
For $d>0$ define
\[
 \mu^d_r(X,\dist)=\inf_{\U\in C_r}\sum_{U\in\U} (\diam(U))^d
\]
and 
\[
 \mu^d(X,\dist)=\lim_{r\to 0}\mu^d_r(X,\dist).
\]

\begin{prop}
\label{propPisoHMeas}
 Let $h\colon (X_1,\dist_1)\to (X_2,\dist_2)$ be a pseudo isometry of compact metric spaces. 
 Then $\mu^d(X_1,\dist_1)=\mu^d(X_2,\dist_2)$ for all $d>0$.
\end{prop}

\begin{proof}
 Given $\epsilon>0$ take $\delta>0$ such that 
 if $x,y\in X_1$ and $\dist_1(x,y)<\delta$ then 
 \[
  \left|\frac{\dist_2(h(x),h(y))}{\dist_1(x,y)}-1\right|<\epsilon.
 \]
 That is
 \[
  (1-\epsilon)\dist_1(x,y)<\dist_2(h(x),h(y))<(1+\epsilon)\dist_1(x,y).
 \]
Therefore, if $\diam_1(U)<\delta$ then 
$\diam_2(h(U))<(1+\epsilon)\diam_1(U)$.
If $r<\delta$ and $\U\in C_r(X_1,\dist_1)$ 
then $h(\U)\in C_{r(1+\epsilon)}(X_2,\dist_2)$. 
Then, for all $\U\in C_r(X_1,\dist_1)$ it holds that  
\[
 \sum_{U\in\U} \diam_2(h(U))^d<(1+\epsilon)^d\sum_{U\in\U} \diam_1(U)^d
\]
This implies that 
\[
 \mu^d_{r(1+\epsilon)}(X_2,\dist_2)\leq (1+\epsilon)^d\mu^d_r(X_1,\dist_1),
\]
and 
\[
 \mu^d(X_2,\dist_2)\leq (1+\epsilon)^d\mu^d(X_1,\dist_1).
\]
Since this holds for all $\epsilon>0$ we conclude that
$\mu^d(X_2,\dist_2)\leq \mu^d(X_1,\dist_1)$.
The other inequality is analogous.
\end{proof}


\subsection{Local form of the metric and transitivity}

Fix a product box $C$ and $p\in C$. 
For $x\in C$ define its coordinates relative to $p$ as
\[
 \begin{array}{l}
x^s=W^s_\expc(p)\cap W^u_\expc(x),\\
x^u=W^u_\expc(p)\cap W^s_\expc(x).
\end{array}
\]
Define a metric 
$\dist_p$ on $C$ by 
\[
 \dist_p(x,y)=\max\{\dist(x^s,y^s),\dist(x^u,y^u)\}
\]
for all $x,y\in C$.

\begin{lem}
If  $|1-x|<\sigma<1$ 
then  
\end{lem}
\begin{equation}
 \label{ecuCHota}
 \left|1-\frac 1x\right|<\frac\sigma{1-\sigma}. 
\end{equation}
\begin{proof}
 We have that $1-x<\sigma<1$. 
 Then $0<1-\sigma<x$.
\end{proof}

\begin{lem}
\label{lemHoloPIso}
 If $f\colon M\to M$ is expansive and $\dist$ is self-similar 
 then the holonomy map on a product box $C\subset M$ is a pseudo isometry.
\end{lem}

\begin{proof}
Taking positive iterates of $C$ we can assume that stable plaques of $C$ have diameter smaller than the 
expansivity constant $\expc$.
Denote by $\pi\colon P^u_1\to P^u_2$ the holonomy of two unstable plaques of $C$. 
Fix $p,q\in P^u_1$. 
By the triangular inequality of the metric we have that
\[
 \begin{array}{ll}
  |\dist(f^n(p),f^n(q))-\dist(f^n(\pi(p)),f^n(\pi(q)))| \leq & \dist(f^n(p),f^n(\pi(p)))+\\
							     & \dist(f^n(\pi(q)),f^n(q))
 \end{array}
\]
for all $n\in\Z$. 
If $n\geq 0$ then 
\[
 |\dist(f^n(p),f^n(q))-\dist(f^n(\pi(p)),f^n(\pi(q)))|\leq
 \frac{\dist(p,\pi(p)) +\dist(\pi(q),q)}{\lambda^n}
  \]
because $p,\pi(p)$ and $q,\pi(q)$ are in stable plaques and $\dist$ is self-similar with expanding factor $\lambda>1$.
If $\dist(p,\pi(p))<\expc$ and $\dist(\pi(q),q)<\expc$ then
\[
 |\dist(f^n(p),f^n(q))-\dist(f^n(\pi(p)),f^n(\pi(q)))|\leq
 \frac{2\expc}{\lambda^n}.
\]
Take $m\geq 0$ such that 
\[
\expc/\lambda^{m+1}<\max\{\dist(p,q),\dist(\pi(p),\pi(q))\}\leq\expc/\lambda^{m}.
\]
Then
\[
\begin{array}{ll}
|\dist(p,q)-\dist(\pi(p),\pi(q))|&= \frac{1}{\lambda^m} \left|\dist(f^m(p),f^m(q))-\dist(f^m(\pi(p)),f^m(\pi(q)))\right|\\
&\leq \frac{1}{\lambda^m}\cdot\frac{2\expc}{\lambda^m}=\frac{2\expc}{\lambda^{2m}}.
\end{array}
\]
Given that
$\expc/\lambda^{m+1}<\max\{\dist(p,q),\dist(\pi(p),\pi(q))\}$ 
we have 
\[
 |\dist(p,q)-\dist(\pi(p),\pi(q))|\leq
 \frac{2\max\{\dist(p,q),\dist(\pi(p),\pi(q))\}}{\lambda^{m-1}}.
\]
Applying (\ref{ecuCHota}) 
we conclude that 
\[
 \left|\frac{\dist(\pi(p),\pi(q))}{\dist(p,q)}-1\right|
 \leq
 \frac{2}{\lambda^{m-1}-2}.
\]
%
%
\end{proof}

\begin{thm}
\label{teoCajaPIso}
If $f\colon M\to M$ is expansive with $\dist$ self-similar, $C\subset M$ is a product box and $p\in C$ 
then $(C,\dist)$ and $(C,\dist_p)$ are pseudo isometric.
\end{thm}

\begin{proof}
 It follows by Lemma \ref{lemHoloPIso} and Theorem \ref{teoDynTri}.
\end{proof}

\begin{rmk}
\azul{There are transitive expansive homeomorphisms with canonical coordinates 
for which the holonomy is not an isometry.
Consider the classical derived from Anosov diffeomorphism on the two-dimensional torus \cite{Smale}. 
Its non-wandering set consists of a fixed point and a basic set $\Omega$ 
that is locally the product of a Cantor set with an arc. 
Let $\gamma$ be a circle embedded in the torus and transverse to the arcs of $\Omega$. 
We have that $\Omega\cap \gamma$ is a Cantor set. 
Following the lines of $\Omega$ we can define a \emph{first return map} $g\colon \Omega\cap\gamma\to\Omega\cap\gamma$. 
If $p,q\in \gamma$ are end points of a \emph{gap} then $\dist(g^n(p),g^n(q))\to 0$ as $n\to\pm\infty$. 
This proves that, independently of the metric (self-similar or not), 
holonomies are not isometries.}
\end{rmk}

We remark that the following result can be applied to Anosov diffeomorphisms of compact (connected) manifolds.

\begin{thm}
\label{thmHoloIsoTrans}
 Let $f$ be an expansive homeomorphism with canonical coordinates 
 of a Peano continuum $M$. 
 If $\dist$ is self-similar and holonomies are isometries then $f$ is transitive. 
\end{thm}

\begin{proof}
Arguing by contradiction, consider from the spectral decomposition a repeller $R\subset \Omega(f)$, an attractor $A\subset \Omega(f)$ 
and a wandering point $x\in M$ such that $f^n(x)\to A$ and $f^{-n}(x)\to R$ as $n\to+\infty$.
From \cite{Nadler}*{Theorem 8.25} we know that Peano continua are locally arc connected. 
Since we have local product structure we have that stable and unstable plaques are locally arc connected. 
Then, there are $N>0$, $y\in R$ and an arc $l\subset W^u_\expc(y)$ from $y$ to $f^{-N}(x)$.

Consider $\delta>0$ 
such that if $\dist(p,q)\leq 2\delta$ then $W^s_\expc(p)\cap W^u_\expc(q)\neq\emptyset$.
Take $z$ in the attractor $A$ such that $f^N(x)\in W^s_\delta(z)$.
Let $\gamma=f^{2N}(l)$, an arc from $f^N(x)$ to $f^{2N}(y)\in R$ contained in the (global) unstable set of $f^N(x)$.
Consider $\gamma$ ordered from $f^N(x)$ to $f^{2N}(y)$ and take $p_0=f^N(x)<p_1<\dots<p_k=f^{2N}(y)$ points in $\gamma$ 
such that $\dist(p_{i-1},p_i)<\delta$ for each $i=1,\dots,k$.
Since $f^N(x)\in W^s_\delta(z)$ we have that $\dist(p_0,z)<\delta$ and then $\dist(p_1,z)<2\delta$. 
Then, 
we can define $q_1=W^u_\expc(z)\cap W^s_\expc(p_1)$. 
Since the holonomy is an isometry we have that $\dist(p_1,q_1)=\dist(f^N(x),z)<\delta$. 
Then, $\dist(p_2,q_1)<2\delta$ and we can define $q_2=W^u_\expc(q_1)\cap W^s_\expc(p_2)$. 
Inductively we obtain a sequence $q_1,\dots,q_k$ as $q_{i+1}=W^u_\expc(q_i)\cap W^s_\expc(p_{i+1})$. 
The point $q_k$ is in $W^u(z)\cap W^s(f^{2N}(y))$. 
That is, $q_k\in A\cap R$. This proves the transitivity of $f$.
\end{proof}

Let us explain why Theorem \ref{thmHoloIsoTrans} is not true if we do not assume that $M$ is locally connected. 
We will consider a subshift of finite type on 4 symbols. 
The transition matrix 
\[
 \left(
 \begin{array}{llll}
  1&1&1&1\\
  1&1&1&1\\
  0&0&1&1\\
  0&0&1&1\\
 \end{array}
 \right)
\]
defines a subshift of finite type $f$ on a Cantor set $M$. 
From \cite{Wa78} we know that subshifts of finite type 
have canonical coordinates (or equivalently, the shadowing property). 
It is clear that $f$ is not transitive because 
the non-wandering set consists on a repeller and an attractor. 
Also, we can consider a self-similar metric with isometric holonomy (the metric given in Example \ref{exSSMShift}). 

\section{The intrinsic measure}
\label{secIntMeas}
In \S \ref{secIntErg} we will apply our results 
to the construction of the intrinsic measure 
of a topologically mixing expansive homeomorphism 
with canonical coordinates. 
In \S \ref{secGeoMarkov} we recall some known facts from \cites{Pesin,AH} that we need.

\subsection{Geometric constructions and Markov partitions}
\label{secGeoMarkov}
Let $\N$ be the set of non-negative integers and
consider a finite set $\U$.
Let $Q\subset \U^\N$ be a topologically mixing subshift of finite type. 

Let $(X,\rho)$ be a compact metric space and 
denote by $2^X$ the set of compact subsets of $X$.
Consider $\Delta\colon Q\times\N\to 2^X$ and $K_*,K^*>0$ such that 
\begin{enumerate}
 \item $\Delta(\omega,n+1)\subset \Delta(\omega,n)$, 
 \item for each $(\omega,n)\in Q\times\N$ there are balls $B_*(\omega,n)$ and $B^*(\omega,n)$ of 
 radius $K_*/\lambda^n$ and $K^*/\lambda^n$ respectively, such that 
 \[
  B_*(\omega,n)\subset\Delta(\omega,n)\subset B^*(\omega,n),
 \]
 \item $\inte(B_*(\omega_1,n))\cap \inte(B_*(\omega_2,n))=\emptyset$ if $\omega_1|_{\{0,1,\dots,n\}}\neq \omega_2|_{\{0,1,\dots,n\}}$,
\end{enumerate}
where $\inte(A)$ denotes the interior of $A$.
Define 
\begin{equation}
 \label{ecuDfFDelta}
  F_\Delta=\cap_{n\geq 0}\cup_{\omega\in Q}\Delta(\omega,n)\subset X.
\end{equation}

\begin{thm}{\cites{Pesin,PW1}}
\label{thmPesin}
In the above conditions, if $d=\ucap(F_\Delta,\rho)$ then
$$0<\mu^d(F_\Delta,\rho)<\infty.$$
\end{thm}

\begin{proof}
 See \cite{Pesin} Theorems 13.1 and 13.4.
\end{proof}

Let $f\colon M\to M$ be an expansive homeomorphism with canonical coordinates.
A closed subset $R\subset M$ is a \emph{rectangle} if $\diam(C)<\delta$, 
$R$ is the closure of its interior and $[x,y]\in R$ for all $x,y\in R$. 
Given $x\in R$ denote by 
$R^u(x)=\{y\in R:[y,x]=y\}.$
Assume that the diameter of the rectangle is so small that
$R^u(x)\subset W^u_\expc(x)$
for all $x\in R$.
\begin{prop}
\label{propRmin}
For every $x\in \inte (R)$ there is $r>0$ such that 
 for all $y\in R^s(x)$ it holds that 
 \[
  W^u_r(y)\subset R^u(y).
 \]
\end{prop}

\begin{proof}
It follows by the compactness of $R$ and the product structure.
\end{proof}

A finite cover of $M$ by rectangles $\U=\{R_1,\dots,R_p\}$ of $M$ is a
\emph{Markov partition} for $f\colon M\to M$ if
\begin{enumerate}
\item $\inte(R_i)\cap\inte(R_j)=\emptyset$ if $i\neq j$,
\item for each $x\in\inte(R_i)\cap f^{-1}(\inte(R_j))$ we have
$f(R_i^s(x))\subset R^s_j(f(x))$ 
and 
$R^u_j(f(x))\subset f(R^u_i(x)).$
\end{enumerate}
Every expansive homeomorphism with canonical coordinates admits Markov partitions by rectangles of
arbitrarily small diameter, see \cite{AH}*{Theorem 4.2.8}.
Define
\[
 Q=\{\omega\in \U^\N:\inte(\omega(j))\cap f^{-1}(\inte(\omega(j+1)))\neq\emptyset\text{ for all }j\geq 0\}.
\]
If $f$ is topologically mixing we can apply \cite{AH}*{Theorem 4.3.5} to obtain that 
$Q$ is a topologically mixing subshift of finite type.
For each rectangle $R_i\in\U$ fix an unstable plaque $P^u_{R_i}\subset R_i$ and define 
\begin{equation}
\label{ecuXCup}
X=\cup_{i=1}^p P^u_{R_i}.
\end{equation}
Given $(\omega,n)\in Q\times\N$
define 
$\Delta(\omega,n)=X\cap (\cap_{j=0}^nf^{-j}(\omega(j)))$.
\begin{rmk}
\label{rmkXFDelta}
A point $x\in X$ is in $\Delta(\omega,n)$ if $f^j(x)$ is in the rectangle $\omega(j)\in\U$ for $j=0,\dots,n$. 
Since $\U$ is a cover of $M$ we conclude that the set $F_\Delta$ defined in (\ref{ecuDfFDelta}) coincides with $X$.
\end{rmk}

\begin{prop}
With the previous notation, it holds that
$\Delta(\omega,n)\subset W^u_{\expc/\lambda^n}(x)$
for all $x\in \Delta(\omega,n)$.
\end{prop}
\begin{proof}
For all $y\in \Delta(\omega,n)$ we know that $f^n(y)\in\omega(n)$. 
Denote by $R_n$ the rectangle $\omega(n)\in\U$.
Since $R_n^u(f^n(y))\subset W^u_\expc(f^n(x))$ we have that $f^n(y)\in W^u_\expc(f^n(x))$. 
Then, $y\in f^{-n}(W^u_\expc(f^n(x)))=W^u_{\expc/\lambda^n}(x)$.
\end{proof}

The following result is based on \cite{Pesin}*{Theorem 22.1}.

\begin{thm}
\label{thmPesin2}
Let $f\colon\Lambda\to\Lambda$ be a 
topologically mixing 
expansive homeomorphism with canonical coordinates of a 
compact metric space $\Lambda$ with a \azul{self-similar metric $\rho$.}
If $d=\ucap(W^u_\expc(x),\rho)$ then
\begin{equation}
\label{ecuMedFinPos}
 0<\mu^d(W^u_\expc(x),\rho)<\infty
\end{equation}
for all $x\in \Lambda$.
\end{thm}

\begin{proof}
For each rectangle $R\in\U$ fix an interior point $x_R$. 
Consider $r_R>0$ from Proposition \ref{propRmin}. 
Define 
$K_*=\min\{r_R:R\in\U\}$ and $K^*=\expc$. 
From Theorem \ref{thmPesin} 
we have that $0<\mu^d(X,\rho)<\infty$, where $X$ is given by (\ref{ecuXCup}).
This is because $F_\Delta=X$ (Remark \ref{rmkXFDelta}). 
Since $f$ is topologically mixing we conclude (\ref{ecuMedFinPos}).
\end{proof}

\subsection{Intrinsic ergodicity}
\label{secIntErg}
Let $f\colon M\to M$ be an expansive homeomorphism with canonical coordinates 
of a compact metric space. 
In addition assume that $f$ is topologically mixing. 
In \cite{Bo71} Bowen (see also \cite{AH}*{Theorem 11.5.13}) 
proved that topologically mixing 
expansive homeomorphisms with canonical coordinates 
have the specification property.
And in \cite{BoUnEqSt} he proved that 
expansive homeomorphisms with specification (on a compact metric space) 
admit a unique measure with maximal entropy. 
A homeomorphism with a unique invariant measure maximizing the entropy is called 
\emph{intrinsically ergodic} \cite{Weiss}. The purpose of this section is 
to show that this special measure can be naturally constructed using a 
self-similar metric. 

Assume that $\dist$ is self-similar with expanding factor $\lambda>1$ 
and define 
$$d=\frac{\ent(f)}{2\log(\lambda)}.$$
Let $\mu$ be the Borel measure on $M$ such that 
given a rectangle $P^s\times P^u\subset M$ it holds that 
\[
 \mu(P^s\times P^u)=\mu^d(P^s)\mu^d(P^s)
\]
where $\mu^d$ is the $d$-dimensional Hausdorff measure. 
From Lemma \ref{lemHoloPIso} we know that the holonomy on a product box is 
a pseudo-isometry and Proposition \ref{propPisoHMeas} proves that pseudo-isometries preserve the Hausdorff measure. 
This means that the measure of a box does not depend on the plaques used to define the box.

\begin{thm}
\label{thmIntMeas}
If $f$ is a topologically mixing 
expansive homeomorphism with canonical coordinates 
of a compact metric space 
then $\mu$ is the measure of maximal entropy.
In particular, $\mu$ does not depend on the self-similar metric.
\end{thm}

\begin{proof}
By Theorem \ref{thmPesin2}, we have that
\[
  0<\mu^d(W^s_\epsilon(x))<\infty\text{ and }
  0<\mu^d(W^u_\epsilon(x))<\infty 
\]
for all $x\in M$ and for all $\epsilon>0$.
Given that $\dist$ is self-similar we can apply \cite{Falconer}*{Scaling property 2.1} to conclude that 
\[
\begin{array}{l} 
\mu^d(f(W^u_\expc(x)),\dist)=\lambda^d\mu^d(W^u_\expc(x),\dist)\\
\mu^d(f(W^s_\expc(x)),\dist)=\lambda^{-d}\mu^d(W^s_\expc(x),\dist)
\end{array}
\]
for all $d>0$.
This implies that $\mu$ is $f$-invariant.

To prove that $\mu$ is the intrinsic measure we apply a result of Bowen explained in \cite{DGS}. 
For this purpose we recall that $\mu$ is $f$-\emph{homogeneous} 
if for all $\epsilon>0$ there are $\delta>0$ and $c>0$ such that 
\[
 \mu(D^n_\delta(y))\leq c\mu(D^n_\epsilon(x))
\]
for all $n\geq 0$ and all $x,y\in M$, where 
$$D^n_\epsilon(x)=\{z\in M:\dist(f^i(x),f^i(z))\leq\epsilon\text{ if }0\leq i\leq n-1\}.$$
According to \cite{DGS}*{Proposition 19.7}, in order to prove that $\mu$ maximizes the entropy it 
is sufficient to show that $\mu$ is $f$-homogeneous. 

From Theorem \ref{teoCajaPIso} we know that for $\epsilon>0$ there is $\delta>0$ such that 
if $\dist(x,y)<\delta$ then 
\[
 (1-\epsilon)\dist_x(x,y)<\dist(x,y)<(1+\epsilon)\dist_x(x,y).
\]
If we define 
\[
 C^n_\delta(x)=[W^s_\delta(x),W^u_{\delta/\lambda^n}(x)]
\]
then
\begin{equation}
 \label{ecuBolas}
 C^n_{\delta/(1+\epsilon)}(x)\subset D^n_\delta(x)\subset C^n_{\delta/(1-\epsilon)}(x).
\end{equation}
Define 
\[
m^\sigma(\epsilon)=\inf\{\mu^d(W^\sigma_\epsilon(x)): x\in M\} 
\]
and 
\[
M^\sigma(\epsilon)=\sup\{\mu^d(W^\sigma_\epsilon(x)): x\in M\}
\]
for $\sigma=s,u$.
We will show that $0<m^\sigma(\epsilon)\leq M^\sigma(\epsilon)<\infty$.
From Theorem \ref{thmPesin2} we know that each 
$\mu^d(W^\sigma_\epsilon(x))$ is positive and finite.
Since $f$ is transitive, we can take $z\in M$ with dense positive orbit. 
For $r>0$ given take $N$ large such that $\{z,f(z),\dots,f^N(z)\}$ is $r$-dense in $M$. 
Then, it is easy to see that $m^s(\epsilon)>\mu^d(W^s_\epsilon(z))\lambda^{-N}>0$. 
The other inequalities are analogous. 
From these inequalities, Equation (\ref{ecuBolas}) and the definition of $\mu$ we have
\[
 m^s(\delta/(1+\epsilon))m^u(\delta/(1+\epsilon))\lambda^{-n}\leq
 \mu(D^n_\delta(x))\leq
  M^s(\delta/(1-\epsilon))M^u(\delta/(1-\epsilon))\lambda^{-n}
\]
for all $x\in M$ and all $n\geq 0$.
If we define 
\[
 c=\frac{M^s(\delta/(1-\epsilon))M^u(\delta/(1-\epsilon))}{m^s(\delta/(1+\epsilon))m^u(\delta/(1+\epsilon))}
\]
we obtain 
\[
 \mu(D^n_\delta(y))\leq c\mu(D^n_\delta(x))
\]
for all $x,y\in M$ and all $n\geq 0$. 
This proves that $\mu$ is $f$-homogeneous and as we explained the proof ends.
\end{proof}

\begin{rmk}If we can apply
\cite{Falconer}*{Corollary 7.4} 
then the Hausdorff dimension of $M$ is 
$\frac{\ent(f)}{\log(\lambda)}$ (assuming the hypothesis of Theorem \ref{thmIntMeas}). 
The problem is that in \cite{Falconer} a global hypothesis (of the book) 
is that $M$ must be contained in Euclidean $\R^n$ (i.e., the metric of $M$ must be induced by an embedding of $M$ in some $\R^n$). 
Whether this hypothesis is essential or not is not clear to the author. 
Note that the finite dimensionality of a compact metric space admiting an expansive homeomorphism 
is proved in \cite{Ma}. This and results from \cite{HW} implies that the space admits a topological embedding in $\R^n$.
\end{rmk}

\end{document}